\documentclass[12 pt]{article}

\usepackage{CJK}
\usepackage{amssymb}
\usepackage{bbm}
\usepackage{amsfonts}
\usepackage{mathrsfs}
\usepackage{graphicx}
\usepackage{amsmath}
\usepackage{amsthm}
\usepackage{xypic}
\usepackage{graphicx}
\usepackage{times}
\usepackage{geometry}
\usepackage{color}
\usepackage{indentfirst}
\usepackage{cases}
\usepackage{hyperref}
\usepackage{graphicx}
\usepackage{float}
\usepackage{amsthm}
\usepackage{amssymb}
\usepackage{amsmath}
\usepackage{mathrsfs}
\usepackage{indentfirst}
\usepackage{geometry}
\usepackage{authblk}
\usepackage{hyperref}

\hypersetup{
    colorlinks=true,
    linkcolor=blue,
    citecolor=blue
}
\newtheorem{Th}{\scshape  Theorem}[section]
\newtheorem{Lem}[Th]{\scshape  Lemma}

\newtheorem{Rem}[Th]{\scshape Remark}

\title{Laplacian eigenvalue distribution and girth of graphs}

\author{Wenhao Zhen\thanks{Corresponding author, E-mail address:zhenwenhao994@163.com.},
 \ \  Dein Wong\thanks{Corresponding author, E-mail address:wongdein@163.com.    Supported by the National Natural Science Foundation of China
(No.12371025).}, \ \ Songnian Xu %707314493@qq.com
\\  {\small  \it  School of Mathematics, China University of Mining and Technology, Xuzhou, China} }

\date{}

\begin{document}
	\maketitle
	
\noindent \textbf{Abstract: }%\large{\textbf{Abstract:}}
{Let $G$ be a connected graph on $n$ vertices with girth $g$.
Let $m_GI$ denote the number of Laplacian eigenvalues of graph $G$ in an interval $I$.
In this paper, we show that if $G$ is not a cycle, then $m_G(n-g+3,n]\leq n-g$. 
Moreover, we prove that $m_G(n-g+3,n]= n-g$ if and only if $G\cong C_3$ or $G\cong K_{3,2}$ or $G\cong U_1$, where $U_1$ is obtained from a cycle by joining a single vertex with a vertex of this cycle.
}


\section{Introduction}
In this paper, all graphs are simple, i.e., they have no loops nor multiple edges.
Let $G$ be a graph with vertex set $V(G)$ and edge set $E(G)$.
The adjacency matrix $A(G)$ of $G$ with $|V(G)|=n$ is an $n\times n$ symmetric matrix whose $(i,j)$ entry   is $1$ if there is an edge between vertex $i$ and vertex $j$, and $0$ otherwise.
Let $N_G(u) = \{v| v \sim u, v \in V(G)\}$ denote the set of neighbors of $u$ in $G$.
The degree of a vertex $u$ in graph $G$, denoted by $d_G(u)$ (or simply $d(u)$), is defined as the number of vertices adjacent to $u$ in $G$.
The Laplace matrix of $G$ is $L(G)=D(G)-A(G)$, where $D(G)$ is the diagonal matrix ${\rm diag} (d(v_1), d(v_2), \dots, d(v_n))$ with $d(v_i)$ the degree of $v_i$, for $i=1,\dots, n$.
It's known that $L(G)$ is a symmetric positive semidefinite matrix and $0$ is one of its eigenvalues.
The eigenvalues of $L(G)$ are called the Laplacian eigenvalues of $G$, and we denote them in non-decreasing order by $0=\mu_n(G)\leq \cdots \leq \mu_1(G)$.
Denote by $m_{G}(\lambda)$ the  multiplicity of $\lambda$ as an eigenvalue of $L(G)$.
Let $m_GI$ denote the number of Laplacian eigenvalues of $G$ in an interval $I$.

It is well known that $m_G[0, n]=n$ for any graph $G$.
The distribution of eigenvalues in an interval $I\subseteq [0,n]$ has attracted considerable attention in the past two decades. 
Many researchers established some bounds of $m_GI$ in terms of different parameters of graphs. 
Grone et al. \cite{Grone2} proved that $m_G[0, 1)\geq q(G)$, where $q(G)$ is the number of quasi-pendant vertices in $G$.
Merris \cite{Merris1} obtained that $m_G(2, n]\geq q(G)$ for connected graph $G$ with $n>2q(G)$.
Guo et al. \cite{Guo1} showed that if $G$ is a connected graph with matching number $m(G)$, then $m_G(2, n]>m(G)$, where $n>2m(G)$.
Recently, Jacobs et al. \cite{Jacobs1} and Sin \cite{Sin1} independently proved that $m_G[0,2-\frac{2}{n})\geq \frac{n}{2}$ if $G$ is a tree of order $n$, which was conjectured in \cite{Trevisan1}.
Ahanjideh et al. \cite{Ahanjideh1} showed that $m_G(n-\alpha(G),n]\leq n-\alpha(G)$ and $m_G(n-d(G)+3, n]\leq n-d(G)-1$, where $\alpha(G)$ and $d(G)$ are the independence number and the diameter of $G$ respectively.
More recently, Xu and Zhou showed that $m_G[n-d(G)+2, n]\leq n-d(G)$ in \cite{Xu1}, which was conjectured in \cite{Ahanjideh1} and $m_G[n-d(G)+1, n]\leq n-d(G)+1$ in \cite{Xu2}.

\begin{figure}
	\centering
	\includegraphics[width=0.3\linewidth]{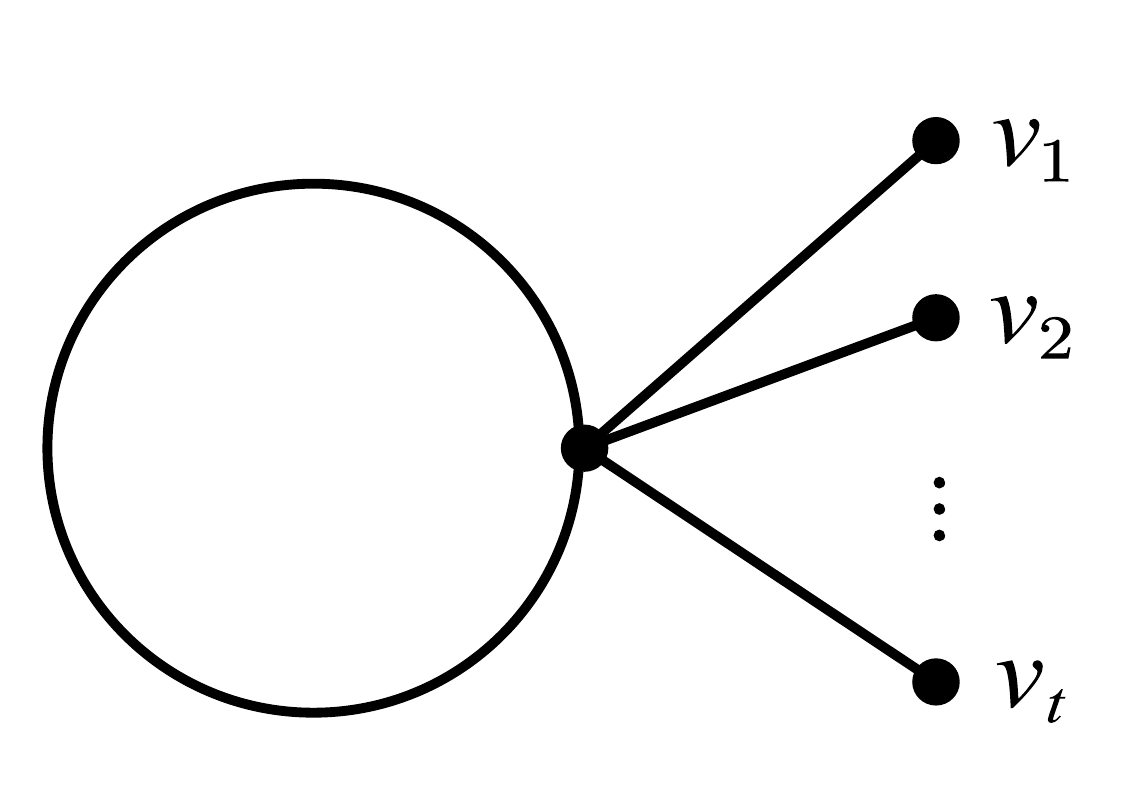}
	\caption{Graph $U_t$.}
	\label{fig:1}
\end{figure}

The girth of a graph $G$, denoted as $g(G)$ ($g$ for short), is defined as the length (number of edges) of a shortest cycle contained in $G$.
If $G$ is acyclic, its girth is conventionally considered to be infinite.
Be inspired by the above works, we consider the bounds of $m_GI$ in terms of girth of a graph.
The unicyclic graph $U_t$ is obtained from a cycle $C$ by joining $t$ pairwise non-adjacent vertices to a vertex of $C$ (See Figure \ref{fig:1}).
In this paper, we show that if $G$ is not a cycle, then $m_G(n-g+3,n]\leq n-g$. 
Moreover, we prove that $m_G(n-g+3,n]= n-g$ if and only if $G\cong C_3$ or $G\cong K_{3,2}$ or $G\cong U_1$.

\section{Preliminaries}
We start with some basic symbols and concepts.
For a subset $W$ of $V(G)$, we denote by $G[W]$ and  $G-W$ the {induced} subgraph of $G$ with vertex set $W$ and $V(G)\backslash W$, respectively.
Let $G$ be a graph, $x$ and $y$ be two vertices in $G$.
The {distance} between $x$ and $y$ in $G$, written as $d_G(x,y)$, is defined as the length of a shortest path between them. 
The chromatic number of $G$, denoted $\chi(G)$, is the smallest positive integer $k$ for which a proper vertex coloring of $G$ using $k$ colors exists. 
We denote by $P_n$ the path with $n$ vertices, $C_n$ the cycle with $n$ vertices, $K_n$ the complete graph with $n$ vertices, and $K_{m_1,\dots,m_t}$ the complete $t$-partite graph with partite sets
of sizes $m_1,\dots,m_t$.

Let $A$ be a Hermitian matrix. We denote by $\rho_k(A)$ the $k$-th largest eigenvalue of $A$ and $\sigma(A) = \{\rho_i(A): i=1,\dots, n\}$ is the spectrum of $A$.
If $\rho$ is an eigenvalue of $A$ with multiplicity $s\geq 2$, then we write it as $\rho^{[s]}$ in $\sigma(A)$.
The spectrum of $L(G)$ is called the Laplacian spectrum of $G$.

We now present several lemmas that are essential for proving our main results.

\begin{Lem}\label{Lem DeL}{\rm (\cite{Mohar1}, Theorem 3.2)}
Let $G = (V,E)$ be a graph with edge set $E$, and let $e \in E$. Then the Laplacian eigenvalues satisfy:
$$\mu_1(G)\geq \mu_1(G - e) \geq \mu_2(G) \nonumber\geq \cdots \geq \mu_{n-1}(G - e) \geq \mu_n(G) = \mu_n(G - e) = 0.$$
\end{Lem}

\begin{Lem}\label{Lem Int}{\rm[Cauchy's interlacing inequality] (\cite{Horn1}, Theorem 4.3.28)}
Let $M$ be a Hermitian matrix of order $n$ and $B$ its principal submatrix of order $p$. Then the eigenvalues satisfy
$\rho_{n-p+i}(M) \leq \rho_i(B) \leq \rho_i(M)$ for $i = 1, \ldots, p$.
\end{Lem}

\begin{Lem}\label{Lem Weyl}{\rm[Weyl's inequalities] (\cite{So1}, Theorem 1.3)}
Let $A$ and $B$ be Hermitian matrices of order $n$. For $1 \leq i, j \leq n$ with $i + j - 1 \leq n$, $$\rho_{i+j-1}(A+B) \leq \rho_i(A) + \rho_j(B)$$ with equality if and only if there exists a nonzero vector $x$ such that $\rho_{i+j-1}(A+B) = (A+B)x$, $\rho_i(A)x = Ax$ and $\rho_j(B)x = Bx$.
\end{Lem}

\begin{Lem}\label{Lem CE}{\rm(\cite{Anderson1})}
(1) If $G$ is the cycle with $n$ vertices, then the Laplacian eigenvalues of $G$ are $4\sin^2(k\pi /n), \, k = 1, 2, \ldots, n.$

(2) If $G$ is the path with $n$ vertices, then the Laplacian eigenvalues of $G$ are $4\sin^2((n-k)\pi/2n), \, k = 1, 2, \ldots, n.$
\end{Lem}

\begin{Lem}\label{Lem D+1}{\rm(\cite{Grone1}, Corollary 2)}
Let $G$ be a graph on $n$ vertices with maximum degree $\Delta \geq 1$. Then $\mu_1(G) \geq \Delta + 1$. For a connected graph $G$ on $n$ vertices, equality holds if and only if $\Delta = n - 1$.
\end{Lem}

\begin{Lem}\label{Lem up}{\rm(\cite{Das1}, Theorem 2.1)}
If $G$ is a graph, then
$$\mu_1 \leq \max\{d(u) + d(v) - |N(u) \cap N(v)| : uv \in E(G)\}.$$
\end{Lem}

\begin{Lem}\label{Lem chromatic number}{\rm(\cite{Wang1}, Theorem 3.1)}
Let $G$ be a connected graph of order $n$ with chromatic number $\chi(G)$. Then $$m_G(n-1,n] \leq \chi(G)-1.$$
\end{Lem}

\begin{Lem}\label{Lem ComP}{\rm(\cite{Ahanjideh1}, Lemma 4.3)}
If $G$ is a complete t-partite graph $K_{r_1, \ldots, r_t}$ with $r_1 + \cdots + r_t = n$ and
$r_1 \leq \cdots \leq r_t$, then its Laplacian spectrum is $\{0, n-r_t^{[r_t-1]}, \dots, n-r_1^{[r_1-1]}, n^{[t-1]}\}$.
\end{Lem}

\section{Proof of the main results}

According to Lemma \ref{Lem CE}, the Laplacian eigenvalues of cycles are well-established; hence, we focus on analyzing the distribution of Laplacian eigenvalues for non-cycle graphs.

\begin{Th}\label{Th up}
Let $G$ be a connected graph of order $n$ with girth $g$. If $G$ is not a cycle, then $$m_G(n-g+3,n]\leq n-g.$$
\end{Th}

\begin{proof}
For $g=3$, we have $m_G(n-g+3, n]= m_G(n, n]=0$, which trivially satisfies the inequality $m_G(n-g+3, n] \leq n - g$. 
Henceforth, we assume $g >3 $.

Let $0=\mu_n(G)\leq \cdots \leq \mu_1(G)$ be the Laplacian eigenvalues of $G$. 
Note that $m_G(n-g+3,n]\leq n-g$ holds if and only if $\mu_{n-g+1}(G)\leq n-g+3$.
We therefore proceed to prove the inequality $\mu_{n-g+1}(G)\leq n-g+3$ below.

Let $C:=v_1\sim v_2\sim \cdots \sim v_g\sim v_1$ be a shortest cycle in $G$.
Since $G$ is a connected graph and not a cycle, the subgraph $G \setminus C$ is non-empty, and edges must exist between $G \setminus C$ and $C$.

\textbf{Case 1.} There does not exist a vertex in $C$ adjacent to all vertices in $G\setminus C $.

Let $H$ be the principal submatrix of $L(G)$ corresponding to the vertices $v_1, \ldots, v_g$.
It's obvious that $$H=L(C)+D,$$ where $D={\rm diag} \{d(v_1)-2, \dots, d(v_g)-2\}$.
By Lemma \ref{Lem Int} and Lemma \ref{Lem Weyl}, we have $$\mu_{n-g+1}(G)=\rho_{n-g+1}(L(G))\leq \rho_{1}(H)\leq \rho_1(L(C))+\rho_1(D).$$
According to Lemma \ref{Lem CE}, we have $\rho_1(L(C))=\mu_1(C)\leq 4$.
Since any vertex in $C$ is not adjacent to all vertices in $G\setminus C$, we know $d(v_i)-2\leq n-g-1$ for $i=1,\dots, g$.
Hence, $$\mu_{n-g+1}(G)\leq \rho_1(L(C))+\rho_1(D)\leq 4+n-g-1=n-g+3.$$

\textbf{Case 2.} There exist a vertex in $C$ adjacent to all vertices in $G\setminus C $.

Without loss of generality, assume $v_1$ adjacent to all vertices in $G\setminus C $.
No two vertices in $G \setminus C$ are adjacent.
Suppose, for contradiction, that there exist $u, v \in G \setminus C$ with $u \sim v$.
Then $u \sim v \sim v_1 \sim u$ forms a 3-cycle, contradicting that $g> 3$.

If there exists a vertex $u \in G \setminus C$ adjacent to at least two vertices on $C$, two neighbors of $u$ on $C$ partition $C$ into two paths, denoted as $P_a$ and $P_b$.
Since $g=a+b+2$,  we have $a+3\geq a+b+2,\ b+3\geq a+b+2$, which implies that $a\leq 1,\ b\leq 1$.
Note that $g>3$, one can get $a=b=1$, $g=4$ and $u$ is adjacent to $v_1$ and $v_3$ in $C$.
Define $V_1 = \{ u \in V(G \setminus C) \mid u \sim v_1 \text{ and } u \sim v_3 \}$
and $V_2 = \{ u \in V(G \setminus C) \mid u \sim v_1 \text{ and } u \not\sim v_3 \}$.
Let $G'$ be a graph obtained by adding edges between all vertices of $V_2$ and $v_3$ (if $V_2=\emptyset$, $G'=G$).
Then $G$ is a subgraph of $G'$ and $G'\cong K_{2,n-2}$.
By Lemma \ref{Lem ComP}, we know the Laplacian spectrum of $G'$ is $\{0, 2^{[n-3]}, n-2, n\}$.
By Lemma \ref{Lem DeL}, we have $m_G(n-g+3,n]=m_G(n-1,n]\leq m_{G'}(n-1,n]=1\leq n-g$.

If every vertices in $G \setminus C$ adjacent to exactly one vertex on $C$.
Then $G\cong U_{n-g}$.
Suppose $v_{g+1}\in G\setminus C$.
Let $H'$ be the principal submatrix of $L(G)$ corresponding to the vertices $v_1, \ldots, v_g, v_{g+1}$ in order. Then $$H'=\begin{pmatrix}
L(C) & 0_{g \times 1} \\
O_{1 \times g} & 0
\end{pmatrix} + M,$$
where  $M = (m_{ij})_{(g+1) \times (g+1)}$ with
$$\begin{cases}
n-g & if\ i = j =1, \\
-1 & if\ (i,j) \in \{(1,g+1), (g+1,1)\}, \\
1 &if\ i = j =g+1, \\
0 & otherwise.
\end{cases}$$
By Lemma \ref{Lem Int} and Lemma \ref{Lem Weyl}, we have $$\mu_{n-g+1}(G)=\rho_{n-(g+1)+2}(L(G))\leq \rho_{2}(H')\leq \rho_1(L(C))+\rho_2(M).$$
Let $n-g=a\geq 1$, one can get $\rho_2(M)=\tfrac{1+a-\sqrt{a^2-2a+5}}{2}$.
According to Lemma \ref{Lem CE}, we have $\mu_{n-g+1}(G)\leq \rho_1(L(C))+\rho_2(M)\leq 4+\tfrac{1+a-\sqrt{a^2-2a+5}}{2}$.
Let $f(a)=a+3-\tfrac{1+a-\sqrt{a^2-2a+5}}{2}-4=\tfrac{a+\sqrt{a^2-2a+5}-1}{2}-1$.
It's easy to see the function $f(a)$ is monotonically increasing for $a \geq 1$.
So, $f(a)\geq f(1)= 0$.
Hence, we have $$\mu_{n-g+1}(G)\leq 4+\frac{1+a-\sqrt{a^2-2a+5}}{2}\leq a+3=n-g+3.$$
\end{proof}

\begin{Rem}\label{Rem cycle}
If $G$ is a cycle, the number of Laplacian eigenvalues in the interval $(n-g+3,n]=(3,n]$ can be explicitly determined. 
Suppose $\mu$ is a Laplacian eigenvalue of $G$. 
By Lemma \ref{Lem CE}, we have $\mu=4\sin^2(\frac{k\pi}{n})$ for some integer $1\leq k\leq n$. 
Hence $\mu>3$ if and only if $4\sin^2(\frac{k\pi}{n})=2(1-\cos(\frac{2k\pi}{n}))>3$, which implies that $\frac{n}{3}<k<\frac{2n}{3}$. 
Therefore, $$m_G(n-g+3,n]=\begin{cases}
\lfloor \frac{2n}{3}\rfloor -\lceil \frac{n}{3}\rceil +1 & if\ n\not\equiv 0(\bmod\ 3), \\
\lfloor \frac{2n}{3}\rfloor -\lceil \frac{n}{3}\rceil -1 & if\ n\equiv 0(\bmod\ 3). 
\end{cases}$$
\end{Rem}

Now, we give a characterization for graphs $G$ with $m_G(n-g+3,n]= n-g$.

\begin{Th}\label{Th ex}
Let $G$ be a connected graph of order $n$ with girth $g$. Then $m_G(n-g+3,n]= n-g$ if and only if $G\cong C_{3}$ or $G\cong K_{3,2}$ or $G\cong U_1$.
\end{Th}

\begin{proof}
Sufficiency:
If $G\cong C_{3}$, we have $n=g=3$ and $m_G(n-g+3,n]=m_G(n,n]=0=n-g$. 
If $G\cong K_{3,2}$, by Lemma \ref{Lem ComP}, we know the Laplacian spectrum of $G$ is $\{0, 2^{[2]}, 3, n \}$.
Hence $$m_G(n-g+3,n]=m_G(4,5]=1=n-g.$$
Let $\Delta$ be the maximum degree of $G$.
If $G\cong U_1$, by Lemma \ref{Lem D+1}, we have $$\mu_1(G)\geq \Delta+1=4$$ with equality holds if and only if $\Delta = n - 1$.
Note that $\Delta =3< n - 1$, we have $\mu_1(G)>4$, which implies that $m_G(n-g+3,n]=m_G(4,5]\geq 1=n-g$.
Moreover, $m_G(n-g+3,n]\leq n-g$ by Theorem \ref{Th up}.
Hence, $m_G(n-g+3,n]= n-g$.

Necessity:
Let $G$ be a connected graph with $m_G(n-g+3,n]= n-g$ and $C:=v_1\sim v_2\sim \cdots \sim v_g\sim v_1$ be a shortest cycle in $G$. 
For $g=3$, we have $m_G(n-g+3, n]= m_G(n, n]=0=n-g$.  
Hence, $n=g=3$, which implies that $G\cong C_3$. 
In the following, we assume that $g>3$. 
Suppose $G$ is a cycle with girth $g>3$, by Remark \ref{Rem cycle}, we have $m_G(n-g+3, n]= m_G(3, n]\geq 1> n-g=0$, a contradiction.
Now we suppose that $G$ is not a cycle and $G\not\cong K_{3,2}$ and $G\not\cong U_{1}$.

\textbf{Claim 1.} For every vertex $x \in G \setminus C$, the distance from $x$ to $C$ is 1.

Suppose there exists a vertex $u$ such that $d(u,C)>1$.
Let $H$ be the principal submatrix of $L(G)$ corresponding to the vertices $v_1, \ldots, v_g, u$ in order.
Then $$H=\begin{pmatrix}
L(C) & 0_{g \times 1} \\
O_{1 \times g} & 0
\end{pmatrix} + D,$$
where $D={\rm diag}\{d(v_1)-2, \dots, d(v_g)-2, d(u)\}$.
Note that $\mid V(G\setminus C)\mid=n-g$, we have $d(v_i)-2\leq n-g-1$ and $d(u)\leq n-g-1$.
By Lemma \ref{Lem Int}, Lemma \ref{Lem Weyl} and Lemma \ref{Lem CE}, we have
\begin{align*}
\mu_{n-g}(G)=\rho_{n-(g+1)+1}(L(G))&\leq \rho_{1}(H)\\
&\leq \rho_1(L(C))+\rho_1(D)\\
&\leq 4+n-g-1\\
&=n-g+3.
\end{align*}
Hence, $m_G(n-g+3,n]\leq n-g-1$, a contradiction.

\textbf{Claim 2.} Every vertex in $G\setminus C$ is adjacent to exactly one vertex on $C$.

Suppose there exists a vertex $u \in G \setminus C$ adjacent to at least two vertices on $C$.
Similarly to Theorem \ref{Th up}'s proof, we get $g = 4$ and $N_C(u)$ is either $\{v_1, v_3\}$ or $\{v_2, v_4\}$.
From Claim 1, all vertices in $G \setminus C$ have non-empty neighborhoods in $C$.
Hence, $1\leq |N_C(v)|\leq 2$ for any vertex $v\in G\setminus C$.
We define:
\begin{align*}
V_1 &= \big\{ v \in G \setminus C \mid N_C(v) \cap \{v_1, v_3\} \neq \emptyset \big\}, \\
V_2 &= \big\{ v \in G \setminus C \mid N_C(v) \cap \{v_2, v_4\} \neq \emptyset \big\}.
\end{align*}
Then $V_1\cap V_2=\emptyset$ and $V_1\cup V_2=V(G\setminus C)$.
Assume $|V_1|=n_1$ and $|V_2|=n_2$.
Let $G'$ be a graph obtained by adding edges between all vertices of $V_1$ and all vertices $V_2$, between all vertices of $V_1$ and all of $\{v_1, v_3\}$ and between all vertices of $V_2$ and all of $\{v_2, v_4\}$.
Then $G'$ is a complete 2-partite graph $K_{n_1+2,n_2+2}$ with partite sets $V_1\cup{\{v_2, v_4\}}$ and $V_2\cup{\{v_1, v_3\}}$.
By Lemma \ref{Lem ComP}, we know the Laplacian spectrum of $G'$ is $\{0, n_1+2^{[n_2+1]}, n_2+2^{[n_1+1]}, n\}$.
By Lemma \ref{Lem DeL}, we have $$m_G(n-g+3,n]=m_G(n-1,n]\leq m_{G'}(n-1,n]=1. $$
Recall that $G\not\cong K_{3,2}$, we have $n-g>1$. 
Hence, $m_G(n-g+3,n]< n-g$, a contradiction.

\textbf{Claim 3.} If $n-g > 2$ and there exist vertices $u, v \in G \setminus C$ with distinct neighbor in $C$, then $u \sim v$.

Suppose, for contradiction, that there exist $u, v \in G \setminus C$ with $N_C(u)=\{v_t\}$, $N_C(v)=\{v_s\}$ ($t\neq s$) and $u \not\sim v$.
Let $H'$ be the principal submatrix of $L(G)$ corresponding to the vertices $v_1, \ldots, v_g, u, v$ in order.
Then $$H'=\begin{pmatrix}
L(C) & 0_{g \times 2} \\
O_{2 \times g} & I_{2 \times 2}
\end{pmatrix}+D'+M,$$
where $D'={\rm diag}\{d(v_1)-2, \dots, d(v_g)-2, d(u)-1, d(v)-1\}$ and $M = (m_{ij})_{(g+2) \times (g+2)}$ with
$$\begin{cases}
-1 & if\ \{i,j\}\in \{\{t,g+1\}, \{s,g+2\}\}, \\
0 & otherwise.
\end{cases}$$
By Lemma \ref{Lem Int}, Lemma \ref{Lem Weyl} and Lemma \ref{Lem CE}, we have
\begin{align*}
\mu_{n-g}(G)=\rho_{n-(g+2)+2}(L(G))&\leq \rho_{2}(H')\\
&\leq \max\{\rho_1(L(C)),1\}+\rho_2(D'+M)\\
&\leq \max\{\rho_1(L(C)),1\}+\rho_1(M)+\rho_2(D')\\
&\leq 4+\rho_1(M)+\rho_2(D').
\end{align*}
By calculation, we have $\rho_1(M)=1$.
Note that $d(u)-1\leq n-g-2$, $d(v)-1\leq n-g-2$, $d(v_i)-2\leq n-g-1$ for $i=1,\dots, g$ and at most one of $\{d(v_i)-2: i=1,\dots, g\}$ is $n-g-1$ when $n-g> 2$.
Hence, $$\mu_{n-g}(G)\leq 4+\rho_1(M)+\rho_2(D')\leq 4+1+n-g-2=n-g+3, $$a contradiction.

\textbf{Claim 4.} For all $u, v \in G \setminus C$, $u \nsim v$.

\begin{figure}
	\centering
	\includegraphics[width=1.0\linewidth]{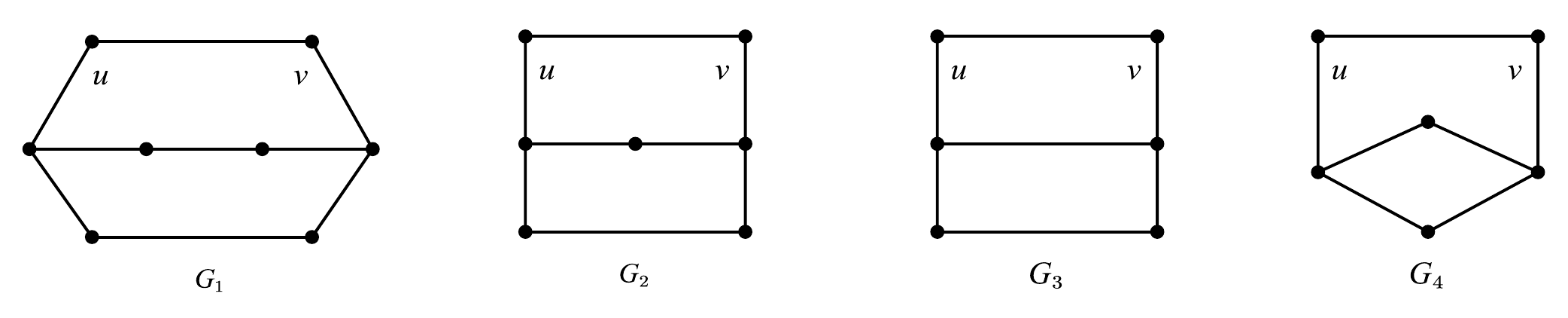}
	\caption{Graphs $G_1,\ G_2,\ G_3$ and $G_4$.}
	\label{fig:2}
\end{figure}

Let $u, v \in G \setminus C$.
If $u, v$ have same neighbor in $C$, then clearly $u \nsim v$; otherwise, a 3-cycle would emerge.
Suppose $N_C(u)=\{v_i\}$, $N_C(v)=\{v_j\}$ ($i\neq j$), and $u\sim v$.
Let $G''$ be the induced subgraph of $G$ with vertex set $V(C)\cup \{u, v\}$, i.e., $G''=G[V(C)\cup \{u, v\}]$.
Note that $v_i$ and $v_j$ partition $C$ into two paths, denoted as $P_a$ and $P_b$ with $a\leq b$ (if $v_i\sim v_j$, then $a=0$).
Since $g=a+b+2$,  we have $a+4\geq a+b+2,\ b+4\geq a+b+2$, which implies that $a\leq b\leq 2$.
Recall that girth $g>3$, one can get $G''$ must be one of $G_1$, $G_2$, $G_3$ or $G_4$ (see Figure \ref{fig:2}). 

If $G''=G_1$ and $n-g>2$, then $g=6$ and $V(G\setminus G'')\neq \emptyset$.
Suppose $w\in V(G\setminus G'')$.
If $N_C(w)\not\subseteq N_C(u)\cup N_C(v)$.
By Claim 3, we know $w\sim u$ and $w\sim v$.
Hence, $w, u, v$ form a 3-cycle, which contradicts $g=6$.
If $N_C(w)\subseteq N_C(u)\cup N_C(v)$.
Without loss of generality, assume $N_C(w)= N_C(u)=\{v_1\}$.
By Claim 3, we know $w\sim v$.
Hence, $w, u, v, v_1$ form a 4-cycle, which contradicts $g=6$.
If $G''=G_1$ and $n-g=2$, i.e., $G=G''=G_1$. 
By Lemma \ref{Lem up}, we have $\mu_1(G)\leq 3+2=5=n-g+3$. 
Hence, we have $m_G(n-g+3,n]=0<n-g$, a contradiction.
Similarly, if $G'' = G_2$, we can again derive a contradiction.

If $G''=G_3$, then $g=4$.
Without loss of generality, assume $N_C(u)=\{v_1\},\ N_C(v)=\{v_2\}$.
Clearly, $N_C(w)\subseteq \{v_1, v_2\}$ for any $w\in V(G\setminus C)$; otherwise, a 3-cycle would emerge.
We define:
\begin{align*}
V_3 &= \big\{ w \in G \setminus C \mid N_C(w)=\{v_1\} \}, \\
V_4 &= \big\{ w \in G \setminus C \mid N_C(w)=\{v_2\} \}.
\end{align*}
Then $V_3\cap V_4=\emptyset$ and $V_3\cup V_4=V(G\setminus C)$.
By Claim 3 and the fact that $u\sim v$, we know all vertices of $V_3$ are adjacent to all vertices of $V_4$.
Suppose that $|V_3|=n_3\geq 1$ and $|V_4|=n_4\geq 1$, then $n=n_3+n_4+4\geq 6$.
As $(n_3+1)I-L(G)$ has $n_4$ equal rows, $(n_3+1)$ is a Laplacian eigenvalue of $G$ with multiplicity at least $n_4-1$.
Similarly, $(n_4+1)$ is a Laplacian eigenvalue of $G$ with multiplicity at least $n_3-1$.
Therefore, the Laplacian spectrum of $G$ must include $0$, $n_3+1^{[n_4-1]}$ and $n_4+1^{[n_3-1]}$. Beyond these $n_3+n_4-1$ eigenvalues, there remain five unknown Laplacian eigenvalues (regardless of whether $n_3=n_4$ or $n_3\neq n_4$).
Clearly, $n_3+1<n-1$ and $n_4+1<n-1$.
Define $S$ as the sum of the five unknown Laplacian eigenvalues.
Note that there are $n_3n_4+n_3+n_4+4$ edges in $G$.
We have $$S+(n_3+1)(n_4-1)+(n_4+1)(n_3-1)=2(n_3n_4+n_3+n_4+4), $$
which implies that $S=2(n_3+n_4)+10=2(n+1)$.
If at least three of these five eigenvalues greater than $n-1$, then $3(n-1)<S=2(n+1)$.
So, $n<5$, which contradicts $n\geq 6$.
If at most two of these five eigenvalues greater than $n-1$, then $m_G(n-g+3,n]=m_G(n-1,n]=n-4\leq 2$.
Hence, we have $n=6$, which implies that $G=G_3$. 
Since $G_3$ is a 2-partite graph, we have $\chi(G)=2$. 
By Lemma \ref{Lem chromatic number}, one can get $m_G(n-g+3,n]=m_G(5,n]\leq \chi(G)-1=1<n-g$, a contradiction. 
Similarly, if $G'' = G_4$, we can get $G=G_4$. 
By Lemma \ref{Lem up}, we have $\mu_1(G)\leq 3+2=5=n-g+3$. 
Hence, we have $m_G(n-g+3,n]=0<n-g$, a contradiction. 

\textbf{Claim 5.} There is exactly one vertex in $G\setminus C$.

Suppose $n-g\geq 2$.
If all vertices in $G\setminus C$ have same neighbor, then $G=U_t$ for some $t\geq 2$.
Suppose that $u\in G\setminus C$ and $u\sim v_1$.
Without loss of generality, we assume the Laplacian matrix $L(G)$ is ordered such that its first row correspond to $v_1$, and the final row correspond to $u$.
Then $$L(G)=\begin{pmatrix}
L(G-u) & 0_{n-1 \times 1} \\
O_{1 \times n-1} & 0
\end{pmatrix} + M',$$
where  $M' = (m_{ij})_{n \times n}$ with
$$\begin{cases}
1 & if\ i = j =1\ or\ i = j =n, \\
-1 & if\ (i,j) \in \{(1,n), (n,1)\}\\
0 & otherwise.
\end{cases}$$
By calculation, we have $\rho_2(M')=0$.
By Lemma \ref{Lem up}, we know $\rho_1(L(G-u))\leq t+3$.
By Lemma \ref{Lem Weyl}, we have
\begin{align*}
\mu_2(G)&=\rho_{2}(L(G))\\
&\leq \rho_1(L(G-u))+\rho_2(M')\\
&\leq t+3\\
&=n-g+3.
\end{align*}
Therefore, $m_G(n-g+3,n]\leq 1< n-g$, a contradiction.

If there exist vertices $u, v \in G \setminus C$ with distinct neighbor in $C$.
Suppose that $N_C(u)=\{v_i\}$ and $N_C(v)=\{v_j\}$, $i\neq j$.
If $n-g>2$, by Claim 3, we know $u\sim v$.
This contradicts with Claim 4.
Hence, $n-g=2$.

If $v_i\not\sim v_j$, we can get $\mu_1(G)\leq 3+2=5=n-g+3$ by Lemma \ref{Lem up}.
Hence, $m_G(n-g+3,n]=0<n-g$, a contradiction.
If $v_i\sim v_j$, we assume that $i=1, j=2$.
Without loss of generality, we assume the Laplacian matrix $L(G)$ is ordered such that its first row correspond to vertices $v_1$, and the second row correspond to vertex $v_2$.
Evidently, $G$ is obtainable from the path $v\sim v_2\sim v_3\cdots\sim v_g\sim v_1\sim u$ by adding edge $v_1v_2$.
Hence, we have $$L(G)=L(P_{g+2})+M'',$$
where  $M''= (m_{ij})_{n \times n}$ with
$$\begin{cases}
1 & if\ i = j =1\ or\ i = j =2, \\
-1 & if\ (i,j) \in \{(1,2), (2,1)\}\\
0 & otherwise.
\end{cases}$$
By Lemma \ref{Lem Weyl} and Lemma \ref{Lem CE}, we have
\begin{align*}
\mu_2(G)&=\rho_{2}(L(G))\\
&\leq \rho_1(L(P_{g+2}))+\rho_2(M'')\\
&< 4\\
&<n-g+3.
\end{align*}
Therefore, $m_G(n-g+3,n]\leq 1< n-g$, a contradiction.

It's easy to see if $G$ is a connected graph that satisfying Claim 1 to Claim 5, then $G$ must be $K_{3,2}$ or $U_1$, which contradicts the initial assumption that $G\not\cong K_{3,2}$ and $G\not\cong U_{1}$.
After all, we know if $G$ is a connected graph with $m_G(n-g+3,n]=n-g$, then $G\cong C_{3}$ or $G\cong K_{3,2}$ or $K\cong U_1$.
\end{proof}

\small {
	
}
\end{document}